\documentclass[10pt]{article}
\usepackage{}
\usepackage{amsfonts}
\usepackage{mathrsfs}
\input{epsf.sty}
\usepackage{epsfig}
\usepackage{amsmath}
\usepackage{amssymb}
\usepackage{amsthm}
\usepackage{times}
\usepackage{float}
\usepackage{color}

\usepackage{indentfirst}
\def\be{\begin{equation}}
\def\ee{\end{equation}}
\def\ba{\begin{array}}
\def\ea{\end{array}}

\newtheorem{thm}{Theorem}[section]

\newtheorem{lem}[thm]{Lemma}
\newtheorem{prop}[thm]{Proposition}

\newtheorem{rem}{Remark}

\numberwithin{equation}{section}%\numberwithin{equation}{section}

\newcommand{\md}{\mathrm{d}}

\def\be{\begin{equation}}
\def\ee{\end{equation}}
\def\br{\begin{eqnarray}}
\def\er{\end{eqnarray}}

\begin{document}
\title{On the Growth in time of Sobolev Norms for Time Dependent Linear Generalized KdV-type Equations \footnote{Supported by NNSFC (No.11790272 and No.11421061) .} }

\author{ $\mbox{Chengming \ Cao}^{\ddag}$ \hspace{12pt}  \  \  Xiaoping Yuan$^{\dag}$\\
${}^{\ddag}\mbox{School of Mathematical Sciences, Fudan University,
Shanghai 200433, P R China}$\\$ \mbox {Email:12110180001@fudan.edu.cn}$\\
$^\dag$ School of Mathematical Sciences, Fudan University,
Shanghai 200433, P R China\\
Email: xpyuan@fudan.edu.cn}

\date{}
\maketitle
 {\bf Abstract}: We give a detailed description in 1-D the growth of Sobolev norms for time dependent linear generalized KdV-type equations on the circle. For most initial data, the growth of Sobolev norms is polynomial in time for fixed analytic potential with admissible growth. If the initial data are given in a fixed smaller function space with more strict admissible growth conditions for $V(x,t)$ , then the growth of previous Sobolev norms is at most logarithmic in time.


\section{Introduction}
When consider the growth in time of Sobolev norms for nonlinear Hamiltonian partial differential equations (PDEs), we can choose the linearized equations of these PDEs to study first. The main example discussed before is the nonlinear Schr\"{o}dinger equation on $\mathbb{T}:=\left[-\pi,\pi\right)$ with periodic boundary conditions as follow
\begin{equation}\label{nls}
  \mathrm{i}u_t=\Delta u+\frac{\partial H}{\partial \bar{u}},
\end{equation}
with Hamiltionian
\begin{equation}\label{ham1}
  H = \int_{\mathbb{T}} \left[|\nabla \phi|^2+F(|\phi|^2)\right],
\end{equation}
where $F$ is a polynomial or smooth function and $\phi(x)=u(0,x)$ is the initial data. Split the data $\phi$ in low and high Fourier modes as
\begin{equation}
  \phi=\phi_1+\phi_2,
\end{equation}
where
\begin{equation}
  \phi_1=\Pi_J\phi = \sum_{|j|\leq J}\hat{\phi}(j)\mathrm{e}^{\mathrm{i}jx}.
\end{equation}

Writing
\begin{equation}
  u = w+v
\end{equation}
with $w$ satisfying the initial value problem
\begin{equation}
  \begin{cases}
     &\mathrm{i}\partial_t{w}=\Delta w+\frac{\partial H}{\partial \bar{w}}\\
     &w(0)=\phi_1
  \end{cases}
\end{equation}
and $v$ satisfying the initial value problem
\begin{equation}
  \begin{cases}
  &\mathrm{i}\partial_t{v}=\Delta v+\frac{\partial^2 H}{\partial \bar{w}\partial w}v+\frac{\partial^2 H}{\partial w^2}\bar{v}+F(w,\bar{w},v,\bar{v})\\
  &v(0)=\phi_2
  \end{cases}
\end{equation}

It can be turned out that $F(w,\bar{w},v,\bar{v})$ is a high order term expected to have a small effect and $(\mathrm{i}\partial_t+\Delta)^{-1}$ has a smoothing effect on the term $\frac{\partial^2 H}{\partial \bar{w}^2}\cdot \bar{v}$. Finally we can use the fact that the flow of the linear equation
\begin{equation}
  \begin{cases}
  &\mathrm{i}\partial_t{v}=\Delta v+\frac{\partial^2 H}{\partial \bar{w}\partial w}v\\
  &v(0)=\phi_2
  \end{cases}
\end{equation}
conserves the $L^2$- norm and has essentially unitary behavior in $H^s$ (up to lower-order error terms) since $\frac{\partial^2 H}{\partial \bar{w}\partial w}$ is real. Therefore, the growth in time of Sobolev norms for the linearized equation can be considered firstly.

Bourgain proved an surprising result in the Appendix of \cite{Bourgain1999Global} using Floquet theory, the linear Schr\"{o}dinger equation of the form
\begin{equation}\label{S}
  \mathrm{i} u_t+\Delta u+ V(x,t)u = 0
\end{equation}
with periodic boundary conditions satisfying polynomial growth, where $V$ is time dependent and smooth.

In 2008 \cite{W2009Logarithmic}, Wang obtain the result of logarithmic growth for the equation (\ref{S}). Essentially, it is there proved that in a period of time, the $H^s$ norm of the high frequencies part is preserved. Besides, using localization properties of eigenfunction, the approximate sulution can be constructed by Floquet solution , then the middle frequencies part is controlled.

For other nonlinear Hamiltonian PDEs, in 1996 \cite{Bourgain1996On}, Bourgain mentioned the generalized KdV-type equations in the periodic case of the form
\begin{equation}\label{K}
  u_t+u_{xxx}+\partial_xf'(u)=0,
\end{equation}
in the remark of his study, where $f$ is sufficiently smooth.

The previous study about growth of Sobolev norms enable us to consider the linearised KdV-type equation,
\begin{equation}\label{equ}
  u_t+u_{xxx}+\frac{1}{2}V_x(x,t)u+V(x,t)u_x=0,
\end{equation}
on $\mathbb{T}:=\left[-\pi,\pi\right)$ with periodic boundary conditions. The potential $V(x,t)$ is time-dependent, bounded and real. We further assume that $V(x,t)$ is real analytic in $(x,t)$ in a strip $D:=(\mathbb{R}+\mathrm{i} \rho)^2 (|\rho|<\rho_0,\rho_0>0)$. Here $\frac{1}{2}$ is used to confirm that the flow of (\ref{equ}) conserves the $L^2$- norm.

Since the nonlinear term $\frac{1}{2}V_x(x,t)u+V(x,t)u_x$ has derivative term $u_x$, it leads to complication in the proof of polynomial growth and the control of high frequencies part. Hence, some admissible growth conditions for $V(x,t)$ is necessary. What is more, If we want to obtain logarithmic growth, then the initial data must be given in a fixed smaller function space with more strict admissible growth conditions for $V(x,t)$, such that the  $H^s$ norm of the high frequencies part in our case is linear growth in a period of time. The method we use to control the $H^s$ norm of the low frequencies part is iteration like \cite{Bourgain1999Global}, which makes the $H^s$ norm of this part is also linear growth. Essentially, in a period of time, the $H^s$ norm of sub-high frequencies in the decomposition of low frequencies part is preserved. With interpolation and the estimate of polynomial growth, we obtain logarithmic growth.

Then, we prove the following result:

\begin{thm}
For all $s>0$ and the initial condition $u_0=u(0)\in H^s(\mathbb{T}),\|u_0\|_{L^2(\mathbb{T})}\leq 1$, there exists $C,\widehat{C}$ such that
\begin{equation}\label{poly}
  \|u(t)\|_{H^s}\leq C^s(s+1)!(|t|^s+1)\|u(0)\|_{H^s}
\end{equation}
where $u(t)$ is the solution to (\ref{equ}), provided that the admissible growth condition imposed on $V(x,t)$
\begin{equation}\label{grcon}
  \int_0^t\|V_x(x,\tau)\|_{\infty} \md \tau \leq \widehat{C}\cdot s\log(|t|+2).
\end{equation}

Moreover, for the initial datum $u_0\in H^{s+1}(\mathbb{T})\cap H^{s}(\mathbb{T})$ satisfying
\begin{equation}\label{scon}
  \|u_0\|_{H^{s+1}}\leq \widetilde{C},
\end{equation}
and
\begin{equation}\label{grcon2}
  \int_0^t\|V_x(x,\tau)\|_{\infty} \md \tau \leq \overline{C}\log(|t|+2),
\end{equation}
there exists $\varsigma>3$ and $C_s$, such that
\begin{equation}\label{log}
  \|u(t)\|_{H^{s}}\leq C_s[\log(|t|+2)]^{\varsigma s}\|u(0)\|_{H^{s}}.
\end{equation}
\end{thm}

\section{Polynomial Growth and Error Estimate}

Let $S(t)$ be the flow for equation (\ref{equ}), we prove $S(t)$ conserves the $L^2$- norm.
\begin{lem}
\begin{equation}\label{con}
  \|S(t)\|_{L^2\rightarrow L^2}=1.
\end{equation}
\end{lem}
\begin{proof}
By integration by part and periodic boundary condition we have
\begin{equation*}
 \begin{aligned}
  &\mathrm{Re}\int_{\mathbb{T}}(u_t+u_{xxx}+\frac{1}{2}V_x(x,t)u+V(x,t)u_x)\bar{u}\md x\\
  =&\mathrm{Re}\int_{\mathbb{T}}\frac{1}{2}\frac{\partial u}{\partial t}|u|^2+\frac{1}{2}V_x|u|^2-\frac{1}{2}V_x|u|^2\md x\\
  =&0,
 \end{aligned}
\end{equation*}
Then we have
\begin{equation*}
  \|S(t)u_0\|_{L^2}=\|u_0\|_{L^2}.
\end{equation*}
\end{proof}
Since $V$ is real analytic in $(x,t)$ and bounded in $D$, we get
\begin{equation}
  \left\|\frac{\partial^mV}{\partial x^m}\right\|_{\infty,\mathbb{T}}\leq C^{m+1}m!,\quad m=0,1,\ldots
\end{equation}
then we have the estimate
\begin{lem}
 \begin{equation}
   \|S(t)\|_{H^s\rightarrow H^s}\leq C^{s}(s+1)!(|t|^s+1).
 \end{equation}
\end{lem}
\begin{proof}
\begin{equation*}
  \begin{split}
     &\frac{\partial}{\partial t}\|u(t)\|_{H^s}^2\\
=&2\mathrm{Re}\left(\frac{\partial^s}{\partial x^s}u(t),\frac{\partial}{\partial t}\frac{\partial^s}{\partial x^s}u(t)\right)\\
=&-2\mathrm{Re}\left( \frac{\partial^s}{\partial x^s}u(t),\frac{\partial^s}{\partial x^s}\left(\frac{1}{2}V_xu+Vu_x\right)\right)\\
=&-2\mathrm{Re}\left( \frac{\partial^s}{\partial x^s}u(t),\frac{1}{2}V_x\frac{\partial^su}{\partial x^s}+V\frac{\partial^{s+1}u}{\partial x^{s+1}}+\frac{1}{2}\sum_{\substack{\gamma+\beta=s\\ \gamma\geq 1}}\frac{\partial^{\gamma+1}V}{\partial x^{\gamma+1}}\frac{\partial^{\beta}u}{\partial x^{\beta}}+\frac{\partial^{\gamma}V}{\partial x^{\gamma}}\frac{\partial^{\beta+1}u}{\partial x^{\beta+1}}\right)\\
=&-2\mathrm{Re}\left( \frac{\partial^s u}{\partial x^s},V_x\frac{\partial^{s}u}{\partial x^{s}}+\frac{1}{2}\sum_{\substack{\gamma+\beta=s\\ \gamma\geq 1}}\frac{\partial^{\gamma+1}V}{\partial x^{\gamma+1}}\frac{\partial^{\beta}u}{\partial x^{\beta}}+\sum_{\substack{\gamma+\beta=s\\ \gamma\geq 2}}\frac{\partial^{\gamma}V}{\partial x^{\gamma}}\frac{\partial^{\beta+1}u}{\partial x^{\beta+1}}\right).\\
   \end{split}
\end{equation*}
It follows that
\begin{equation}\label{base}
  \frac{\partial}{\partial t}\|u(t)\|_{H^s}\leq\|V_x\|_{\infty}\|u(t)\|_{H^s}+\sum_{\substack{\gamma+\beta=s\\ \gamma\geq 1}}C^{\gamma+2}(\gamma+1)!\|u\|_{H^\beta}+\sum_{\substack{\gamma+\beta=s\\ \gamma\geq 2}}C^{\gamma+1}\gamma!\|u(t)\|_{H^{\beta+1}}
\end{equation}
Using the following interpolation in (\ref{base}):
\begin{equation}\label{intp}
  \|u(t)\|_{H^{s-\gamma}}\leq\|u(t)\|_{H^{s}}^{\frac{s-\gamma}{s}}\|u(t)\|_{L^2}^{\frac{\gamma}{s}}\quad(s>\gamma)
\end{equation}
we have
\begin{equation}
  \begin{split}
  \frac{\partial}{\partial t}\|u(t)\|_{H^s}\leq&\|V_x\|_{\infty}\|u(t)\|_{H^s}+C^3\|u(t)\|_{H^s}^{1-\frac{1}{s}}\|u(t)\|_{L^2}^\frac{1}{s}+C^4\|u(t)\|_{H^s}^{1-\frac{2}{s}}\|u(t)\|_{L^2}^\frac{2}{s}\\
  +&\ldots+C^{\gamma+2}(\gamma+1)!\|u\|_{H^s}^{1-\frac{\gamma}{s}}\|u(t)\|_{L^2}^{\frac{\gamma}{s}}+\ldots+C^{s+2}(s+1)!\|u(t)\|_{L^2}.
  \end{split}
\end{equation}
In view of the admissible growth condition for $V(x,t)$ (\ref{grcon}) and the equation
\begin{equation}
  \frac{\partial}{\partial t}\|u(t)\|_{H^s}^{\frac{\gamma}{s}}=\frac{\gamma}{s}\|u(t)\|_{H^s}^{\frac{\gamma}{s}-1}\frac{\partial}{\partial t}\|u(t)\|_{H^s}\quad(1\leq \gamma \leq s),
\end{equation}
we obtain that
\begin{equation}
  \|u(t)\|_{H^s}\leq C^s(s+1)!(|t|^s+1)\|u(0)\|_{H^s},
\end{equation}
with a larger $C$.
\end{proof}
\begin{rem}
Since $\forall s>0$, we have the polynomial growth estimate, then on the $\mathbb{T}$, for $u_0\in H^{s+1}(\mathbb{T})\bigcap H^s(\mathbb{T})$,
\begin{equation}\label{s+1}
  \|u(t)\|_{H^{s+1}}\leq C^s(s+1)!(|t|^{s+1}+1)\|u(0)\|_{H^{s+1}}.
\end{equation}
\end{rem}

Next, we give an error estimate in a period of time $T$.

Let $\tilde{\phi}\in C^\infty_0[-\pi,\pi]$ be a Gevrey function of order $\alpha$:
\begin{equation}
  \max_{\tau\in[-\pi,\pi]}\left|\frac{\partial^m\tilde{\phi}(\tau)}{\partial\tau^m}\right|\leq C^{m+1}(m!)^\alpha,\quad 1<\alpha<\infty
\end{equation}
satisfying
\begin{equation}\left\{
  \begin{split}
     & 0\leq \tilde{\phi}\leq 1,\\
     & \tilde{\phi}(\tau)=1,\quad |\tau|\leq 1\\
     & \tilde{\phi}(\tau)=0,\quad |\tau|\geq \pi,
  \end{split}\right.
\end{equation}

Let
\begin{equation}
  \phi(t)=\tilde{\phi}\left(\frac{t}{T}\right).
\end{equation}
Define
\begin{equation}
  V_1(x,t)=\sum_{j\in\mathbb{Z}}V(x,T+2\pi j T)\phi(t+2\pi j T).
\end{equation}
Hence $ V_1((x,t)$ is $2\pi$ periodic in $x$, $2\pi T$ periodic in $t$, analytic in $x$, Gevrey in $t$ of order $\alpha,(1<\alpha<\infty)$.
\begin{equation}
   V_1(x,t)=V(x,t),\quad  \partial_xV_{1}(x,t)=\partial_xV(x,t).\quad (|t|\leq T)
\end{equation}

Let $u_1(t)$ be the solution to
\begin{equation}\label{equ2}
  u_t+u_{xxx}+\frac{1}{2}\partial_xV_{1}(x,t)u+V_1(x,t)u_x=0
\end{equation}
with the initial condition $u_1(0)=u_0$, then $u_1(t)=u(t)$, the flow $S(t)$ for (\ref{equ}) can also be used to describe (\ref{equ2}) in $|t|\leq T$.

Moreover, we have the following error estimate,
\begin{lem}
For $|t|\leq T$, there exists $\bar{u}$ and $\zeta$ such that
\begin{equation*}
  \left(\partial_t+\partial_{xxx}+\frac{1}{2}\partial_xV_{1}(x,t)+V_1(x,t)\partial_x\right)\bar{u}=\zeta,
\end{equation*}
where $\bar{u}(0)=u_0$ and $\|\zeta\|_{L^2}\leq \varepsilon(T)$ for all $|t|\leq T$. Then we have the error estimate
\begin{equation}\label{ero}
  \|\bar{u}-u_1\|_{L^2}<\varepsilon(T)|t|\leq \varepsilon(T)T
\end{equation}
for $|t|\leq T$.
\end{lem}
\begin{proof}
  It follows from (\ref{con}) and
  \begin{equation*}
    (\bar{u}-u_1)(t)=\int_0^tS(t)S(\tau)^{-1}\zeta(\tau)\md \tau.
  \end{equation*}
\end{proof}

\section{Floquet Solutions and Localization Property}
From error estimate (\ref{ero}) we know that if
\begin{equation}
  \varepsilon(T)\leq \frac{1}{T^{\eta}}=\mathrm{e}^{-\eta\log T}\quad(\eta>1)
\end{equation}
or decay in a more rapid rate,
\begin{equation}
  \varepsilon(T)\leq \mathrm{e}^{-(\log T)^{\eta}},
\end{equation}
which can overcome other polynomial growth we meet later. Then we can use $\bar{u}$ to give an estimate in the approximation process.

Since we consider solutions for finite time $|t|\leq T$, we replace $V_1$ by $V_2$ defined as
\begin{equation}\label{v2}
  V_2(x,t)=\sum_{\substack{|j|\leq(\log T)^\sigma\\|n|\leq T(\log T)^\sigma}}\widehat{V}_1(j,n)\mathrm{e}^{\mathrm{i}(jx+\frac{n}{T}t)},
\end{equation}
with the Fourier transform of $V_1$
\begin{equation}
  \widehat{V}_1(j,n)=\int_{-\pi T}^{\pi T}\int_{-\pi}^{\pi}V_1(x,t)\mathrm{e}^{-\mathrm{i}jx}\mathrm{e}^{-\mathrm{i}\frac{n}{T}x}\md x\md t
\end{equation}
satisfying
\begin{equation}\label{decay}
  \begin{split}
      |\widehat{V}_1(j,n)|\leq & C\mathrm{e}^{-c|j|},\quad |j|\geq (\log T)^\delta, \\
      \leq& C\mathrm{e}^{-c|\frac{n}{T}|^{1/\alpha}},\quad |n|\geq T(\log T)^\delta,\\
      \quad &(T\gg 1,0<C,c,\delta<\infty),
  \end{split}
\end{equation}
where $\sigma>2\alpha+\delta>\alpha+\delta>2$. Then
\begin{equation}\label{v11}
  \|V_1-V_2\|_{\infty}\leq \mathrm{e}^{-(\log T)^{\sigma'/\alpha}}\ll\frac{1}{T^p},
\end{equation}
\begin{equation}\label{v12}
  \|\partial_xV_1-\partial_xV_2\|_{\infty}\leq \mathrm{e}^{-(\log T)^{\sigma''/\alpha}}\ll\frac{1}{T^p},
\end{equation}
provided $p<(\log T)^{\sigma''/\alpha-1}$, where $\sigma>\sigma'>\sigma''>2\alpha+\delta$.

For $|t|\leq T$, The approximation (\ref{v11}) and (\ref{v12}) permit us to use Floquet solution to
\begin{equation}\label{equ3}
  u_t+u_{xxx}+\frac{1}{2}\partial_xV_{2}(x,t)u+V_2(x,t)u_x=0.
\end{equation}
Since (\ref{equ3}) is time periodic with period $2\pi T$, any $L^2$ solution can be written as a linear superposition of Floquet solutions of the form
\begin{equation}
  \mathrm{e}^{\mathrm{i}Et}\check{\xi}(x,t),
\end{equation}
where $\check{\xi}(x,t)$ is $2\pi$ periodic in $x$ and $2\pi T$ periodic in $t$:
\begin{equation}
  \check{\xi}(x,t)=\sum_{(j,n)\in \mathbb{Z}^2}\xi(j,n)\mathrm{e}^{\mathrm{i}(jx+\frac{n}{T}t)}
\end{equation}
$E$ is called the Floquet eigenvalue; $E,\xi$ satisfy the eigenvalue equation:
\begin{equation}\label{eigen}
  H\xi = \left[\mathrm{diag}\left(j^3-\frac{n}{T}\right)-A\ast-B\ast\right]\xi = E\xi
\end{equation}
on $\ell^2(\mathbb{Z}^2)$, where $\ast$ denotes convolution:
\begin{equation}
  (A\ast\xi)(j,n)=\frac{1}{2}\sum_{(j',n')\in \mathbb{Z}^2}(j-j')\widehat{V}_2(j-j',n-n')\xi(j',n')
\end{equation}
\begin{equation}
  (B\ast\xi)(j,n)=\sum_{(j',n')\in \mathbb{Z}^2}j'\widehat{V}_2(j-j',n-n')\xi(j',n')
\end{equation}
\begin{equation}\label{v2decay}
   \begin{split}
      \widehat{V}_2(j,n)=&\widehat{V}_1(j,n) \quad \text{if} |j|\leq (\log T)^{\sigma} \text{and} |n|\leq T(\log T)^{\sigma}, \sigma>2\alpha+\delta>2, \\
      =& 0\quad \text{otherwise},
   \end{split}
\end{equation}
and $\widetilde{V}_1$ satisfies (\ref{decay}).

We identify the initial condition $\hat{u}_0\in \ell^2(\mathbb{Z})$ with $\tilde{u}_0\in \ell^2(\mathbb{Z}^2)$, where
\begin{equation}
  \left\{
  \begin{split}
     & \tilde{u}_0(j,0)=\hat{u}_0(j)\\
     & \tilde{u}_0(j,n)=0,\quad n\neq 0.
  \end{split}
  \right.
\end{equation}
Since we only concerned about finite time: $|t|\leq T$, it is sufficient to solve the eigenvalue problem in (\ref{eigen}) in a finite region
\begin{equation}\label{reg}
  \Lambda=\{(j,n)\in\mathbb{Z}^2||j|\leq J(T),|n|\leq DT(\log T)^{\sigma}\},
\end{equation}
where $J(T)>T^s$ depending on $T$ and the Sobolev index $s,D>2\pi$ as in the following proposition, $\sigma>2\alpha+\delta>2$ as in (\ref{v2}).

For any subset $\Phi\subset \mathbb{Z}^2$, define $H_\Phi$ to be the restriction of $H$ to $\Phi$:
\begin{equation}\label{Hlam}
  H_\Phi(n,j;n',j')=\left\{
  \begin{split}
     & H(n,j;n',j'),& (n,j)\text{and} (n',j')\in \Phi\\
     & 0 &\text{otherwise}.
  \end{split}\right.
\end{equation}

We have the following localization property on eigenfunctions of $H_{\Lambda}$.
\begin{prop} Assume
\begin{equation}
  H_{\Lambda}\xi=E\xi,\|\xi\|_{\ell^2(\Lambda)}=1.
\end{equation}
Define
\begin{equation}
  \Omega_0 = \{(j,n)\in \Lambda ||j|\leq 4D(\log T)^\sigma\},\quad(\sigma>2\alpha+\delta>2)
\end{equation}
and for any $(j_0,n_0)\in \Lambda$, define
\begin{equation}
  \Omega'(j_0,n_0) = \{(j,n)\in \Lambda |\left||j|-|j_0|\right|\leq (\log T)^\sigma,|n-n_0|\leq T(\log T)^\sigma\},\quad(\sigma>2\alpha+\delta>2),
\end{equation}
Then for all $\xi$ eigenfunctions of $H_\Lambda$ as in (\ref{Hlam}), $\xi$ satisfies either
\begin{equation}\label{xid1}
  \|\xi\|_{\ell^2(\Lambda\setminus\Omega_0)}\leq \mathrm{e}^{-(\log T)^{(\frac{\sigma''-\sigma}{\alpha})}}
\end{equation}
or
\begin{equation}\label{xid2}
  \|\xi\|_{\ell^2(\Lambda\setminus\Omega')}\leq \mathrm{e}^{-(\log T)^{(\frac{\sigma''-\sigma}{\alpha})}}\quad(\sigma>\sigma''>2\alpha+\delta>2).
\end{equation}
for some $\Omega'=\Omega'(j_0,n_0),(j_0,n_0)\in \Lambda$.
\end{prop}
\begin{proof}
Since
\begin{equation}
  (A\ast+B\ast)(\xi)(j,n)=j\sum_{(j',n')\in \mathbb{Z}^2}\widehat{V}_2(j-j',n-n')\xi(j',n')-A\ast\xi(j,n),
\end{equation}
for any given $E$, the resonant set $\Omega$ can be defined as
\begin{equation}\label{leq}
  \left|j^3 +C_Vj-\frac{n}{T}-E\right|\leq (\log T)^\sigma,\quad(\sigma>2\alpha+\delta>2),
\end{equation}
if $(j,n)\in \Omega$, where $|C_V|\leq 2\pi\|V_2\|_{\infty}$.

Then
\begin{equation}\label{inver}
  \|(H_{\Lambda \setminus \Omega}-E)^{-1}\|\leq\frac{1}{(\log T)^\sigma-2\pi\|\partial_xV_2\|_{\infty}}\leq \frac{2}{(\log T)^\sigma},
\end{equation}
if
\begin{equation*}
  \|V_2\|_{\infty}<\frac{1}{4\pi}(\log T)^\sigma,\|\partial_xV_2\|_{\infty}<\frac{1}{4\pi}(\log T)^\sigma.
\end{equation*}
Considering that
\begin{equation*}
  \left|\frac{n}{T}\right|\leq D(\log T)^\sigma,|C_V|\leq 2\pi(\log T)^\sigma,
\end{equation*}
we obtain the following two result:

(\textrm{i}) $|E|\leq 10 D^3(\log T)^{3\sigma}$

This leads to $|j|\leq 3D(\log T)^{\sigma}$. So $\Omega\subset \{(j,n)\in\Lambda||j|\leq 3D(\log T)^{\sigma}\}$. Define $B = \Lambda\setminus \Omega,B_0 = \Lambda\setminus \Omega_0, B_0\subset B$. Let $P_B,P_{B_0}$ be projections onto the sets $B,B_0$.

Assume $\xi$ is an eigenfunction with eigenvalue $|E|\leq 10 D^3(\log T)^{3\sigma}$. Then

\begin{equation}\label{eigen1}
  P_B\xi=-(H_B-E)^{-1}P_B\Gamma\xi
\end{equation}
where
\begin{equation}
  \Gamma = H_\Lambda-H_B\oplus H_\Omega.
\end{equation}
So
\begin{equation}
  P_B\xi=-(H_B-E)^{-1}P_B\Gamma P_\Omega\xi.
\end{equation}
Let
\begin{equation}
  \Gamma_0 = H_B-H_B\oplus H_{B\setminus B_0}.
\end{equation}
Then
\begin{equation}
  \begin{split}
     P_{B_0}\xi &= P_{B_0}P_B\xi  \\
     &=-P_{B_0}(H_{B_0}-E)^{-1}P_{B_0}\Gamma P_{\Omega}\xi+P_{B_0}(H_{B_0}-E)^{-1}\Gamma_{0}(H_{B}-E)^{-1}P_B\Gamma P_\Omega\xi,
  \end{split}
\end{equation}
where we used $B_0\subset B$.

Observing that the eigenvalue problem is considered in a finite region $\Lambda$ (\ref{reg}) and the decay of Fourier coefficients (\ref{decay}) (\ref{v2decay}), we have
\begin{equation}\label{offd}
  \|\Gamma \xi\|_{\ell^2}\leq\mathrm{e}^{-c(\log T)^{\sigma-\delta}},\|\Gamma_0 \xi\|_{\ell^2}\leq\mathrm{e}^{-c(\log T)^{\sigma-\delta}}.
\end{equation}
Using (\ref{inver}) on $(H_{B_0}-E)^{-1}$ and $(H_{B}-E)^{-1}$, we obtain
\begin{equation}\label{eigen2}
  \|P_{B_0}\xi\|_{\ell^2}\leq\frac{4\mathrm{e}^{-c(\log)^{\sigma-\delta}}}{(\log T)^{2\sigma}}<\mathrm{e}^{-(\log T)^{\frac{\sigma''-\delta}{\alpha}}}\quad(\sigma>\sigma''>2\alpha+\delta, T\gg 1),
\end{equation}
which is (\ref{xid1}).

(\textrm{ii}) $|E|> 10 D^3(\log T)^{3\sigma}$

This gives
\begin{equation}\label{j}
  |j|\geq 2D(\log T)^{\sigma}.
\end{equation}
So if there exist $(j,n),(j',n')\in \Omega\subset \Gamma, |j|\neq |j'|$, then
\begin{equation*}
  \left|\frac{n-n'}{T}+j^2-j'^2\right|\leq 2(\log T)^\sigma
\end{equation*}
from (\ref{leq}). Using (\ref{j}), this implies
\begin{equation*}
  \begin{split}
     & \left|\frac{n-n'}{T}\right|\geq(|j|+|j'|)(|j|-|j'|)-2(\log T)^\sigma \\
     & \geq (4D-2)(\log T)^\sigma>2D(\log T)^{\sigma}
  \end{split}
\end{equation*}
which is a contradiction from the definition of $\Lambda$. So $|j|=|j'|$ and
\begin{equation}
  \left|\frac{n-n'}{T}\right|\leq 2(\log T)^\sigma < 2D(\log)^\sigma
\end{equation}
for $D>2\pi$, if both $(j,n),(j',n')\in \Omega$. (\ref{xid2}) follows by using the same argument as in (\ref{eigen1})-(\ref{eigen2}) with $\Omega'$ replacing $\Omega_0$.
\end{proof}
\begin{lem}
Let $\chi_S$ be the characteristic function of the set $S$:
\begin{equation}
  \chi_S|_S=1,\chi_S|_{\Lambda\setminus S}=0.
\end{equation}
For an eigenfunction $\xi$ satisfying the localization property (\ref{xid2}), let
\begin{equation}\label{pri}
  \xi'(j,n)=\chi_{\Omega'}\xi(j,n),
\end{equation}
then
\begin{equation}\label{appf}
  \mathrm{e}^{\mathrm{i}Et}\check{\xi'}(x,t):=\mathrm{e}^{\mathrm{i}Et}\sum_{(j,n)\in\Omega'}\xi'(j,n)\mathrm{e}^{\mathrm{i}(jx+\frac{n}{T}t)}
\end{equation}
is an approximate Floquet solution of (\ref{equ2}) satisfying the error estimate for $|t|\leq T$
\begin{equation}\label{ero2}
  \|\mathrm{e}^{\mathrm{i}Et}\check{\xi'}(t)-S(t)\check{\xi'}(0)\|_{L^2}\leq T\mathrm{e}^{-\frac{7}{12}(\log T)^{\left(\frac{\sigma''-\delta}{\alpha}\right)}}, \quad(\sigma''>2\alpha+\delta>2)
\end{equation}
where $S(t)$ is the flow of (\ref{equ2}).
\end{lem}
\begin{proof}
Since $\Lambda$ (\ref{reg}) is a finite region, $\|H_{\Lambda-E}\|_{\ell^2\rightarrow\ell^2}$ is controlled by the polynomial growth about $T$, with localization property of $\xi$ (\ref{xid2}), we have
\begin{equation}\label{lam}
  \begin{split}
    \|(H_{\Lambda}-E)\xi'\|_{\ell^2(\Lambda)}\leq &\|H_{\Lambda}-E\|_{\ell^2\rightarrow\ell^2}\|\xi\|_{\ell^2(\Lambda\setminus \Omega')} \\
    \leq & \mathrm{e}^{-\frac{2}{3}(\log T)^{(\frac{\sigma''-\delta}{\alpha})}} ,(\sigma''>2\alpha+\delta>2).
  \end{split}
\end{equation}
Define
\begin{equation}
   \begin{split}
     \widetilde{H}=&\mathrm{diag}\left(\frac{n}{T}-j^3\right)+\widetilde{A}\ast+\widetilde{B}\ast\\
     =&\mathrm{diag}\left(\frac{n}{T}-j^3\right)+(\widetilde{A}-A)\ast+(\widetilde{B}-B)\ast,
   \end{split}
\end{equation}
where
\begin{equation}
  ((\widetilde{A}-A)\ast\xi')(j,n)=\frac{1}{2}\sum_{(j',n')\in \mathbb{Z}^2}(j-j')(\widehat{V}_1-\widehat{V}_2)(j-j',n-n')\chi_{\Omega'}\xi(j',n'),
\end{equation}
\begin{equation}
  ((\widetilde{B}-B)\ast\xi')(j,n)=\sum_{(j',n')\in \mathbb{Z}^2}j'(\widehat{V}_1-\widehat{V}_2)(j-j',n-n')\chi_{\Omega'}\xi(j',n').
\end{equation}
Then
\begin{equation}
  (\widetilde{H}-E)\xi'=(H_\Lambda-E)\xi'+\widetilde{\Gamma}\xi'+(\widetilde{A}-A)\ast\xi'+(\widetilde{B}-B)\ast\xi'
\end{equation}
where
\begin{equation}
  \widetilde{\Gamma}=H-H_\Lambda\oplus H_{\Lambda^c}
\end{equation}
Using (\ref{v11})(\ref{v12}) and (\ref{lam}), we obtain
\begin{equation}
  \|(\widetilde{H}-E)\xi'\|_{\ell^2}\leq 2\mathrm{e}^{-\frac{2}{3}(\log T)^{(\frac{\sigma''-\delta}{\alpha})}} ,(\sigma''>2\alpha+\delta>2).
\end{equation}
Taking (\ref{appf}) into (\ref{equ2}), denoted $\mathrm{e}^{\mathrm{i}Et}\check{\xi'}$ by $\bar{u}$ with initial datum $\bar{u}(0)=\check{\xi'}(x,0)$, then
\begin{equation}
  \begin{split}
     \zeta:=&\left(\partial_t+\partial_{xxx}+\frac{1}{2}\partial_xV_{1}(x,t)+V_1(x,t)\partial_x\right)\bar{u}\\
     =&-\mathrm{i}\mathrm{e}^{\mathrm{i}Et}\sum_{(j,n)\in\mathbb{Z}^2}\left((\widetilde{H}-E)\chi_{\Omega'}\xi\right)\mathrm{e}^{\mathrm{i}(jx+\frac{n}{T}t)}.
  \end{split}
\end{equation}
By Plancherel's identity on $\mathbb{T} $ denoted by $\left[-\pi , \pi \right)$ with periodic boundary condition,
\begin{equation}
\begin{split}
     \|\zeta\|_{L^2(\mathbb{T})}=&\left\|\left(\partial_t+\partial_{xxx}+\frac{1}{2}\partial_xV_{1}(x,t)+V_1(x,t)\partial_x\right)\bar{u}\right\|_{L^2(\mathbb{T})}\\
     \leq&(2DT(\log T)^\sigma+1)\|(\widetilde{H}-E)\xi'\|_{\ell^2}\\
     \leq&\mathrm{e}^{-\frac{7}{12}(\log T)^{(\frac{\sigma''-\delta}{\alpha})}}
\end{split}
\end{equation}
and \textbf{Lemma 2.3.}, we have (\ref{ero2}).
\end{proof}
\section{Estimate of Sobolev Norms on Intermediate Frequencies}
Let $\Pi_J$ be the Fourier multiplier such that
\begin{equation}\label{Pi}
\widehat{\Pi}_J(j)=\left\{
  \begin{split}
     & 1 \quad |j|\leq J/2, \\
     & 2(1-|j|/J) \quad J/2<|j|\leq J,\\
     & 0 \quad j>|J|.
  \end{split}\right.
\end{equation}
Assume that $1<s\leq \log T$, $\|u_0\|_{H^s}=1$  and
\begin{equation}\label{TD}
  J=T^{10s}.
\end{equation}
Then we consider the Sobolev norms on middle frequency
\begin{equation}\label{midf}
  \|\Pi_{J/4}S(t)(\Pi_{J/2}-\Pi_{2J_0})u_0\|_{H^s}
\end{equation}
after two cut-offs below.
\begin{equation}\label{cutoff}
 \begin{split}
  \|S(t)u_0\|_{H^s}&\leq\|\Pi_{J/4}S(t)u_0\|_{H^s}+\|(1-\Pi_{J/4})S(t)u_0\|_{H^s},\\
  &\leq\|\Pi_{J/4}S(t)\Pi_{2J_0}u_0\|_{H^s}\\
  &+\|\Pi_{J/4}S(t)(\Pi_{J/2}-\Pi_{2J_0})u_0\|_{H^s}\\
  &+\|\Pi_{J/4}S(t)(I-\Pi_{J/2})u_0\|_{H^s}+\|(1-\Pi_{J/4})S(t)u_0\|_{H^s},
 \end{split}
\end{equation}
where
\begin{equation}\label{dj0}
  J_0=4D(\log T)^\sigma\quad(\sigma>2\alpha+\delta>2).
\end{equation}
To estimate(\ref{midf}), we need the following lemma
\begin{lem}
Denoted $(\Pi_{J/2}-\Pi_{2J_0})u_0$ by $\phi$ for short, then for
\begin{equation}
  \mathrm{supp}\hat{\phi}\subseteq [-2J,-J_0/2]\cup[J_0/2,2J],
\end{equation}
we have
\begin{equation}\label{middle}
  \|\Pi_{J/2}S(t)\phi\|_{H^s}\leq C^s \|\phi\|_{H^s}.
\end{equation}
\end{lem}
\begin{proof}
   Define $\tilde{\phi}$ as
\begin{equation}
  \tilde{\phi}=\left\{
  \begin{split}
     & \tilde{\phi}(j,0)=\hat{\phi}(j), \\
     & \tilde{\phi}(j,n)=0,\quad n\neq 0.
  \end{split}\right.
\end{equation}
then $\mathrm{supp} \tilde{\phi}\subset \Lambda$, where $\Lambda$ is defined in (\ref{reg}). $|\Lambda|\leq T^{10s+2}$. $\tilde{\phi}\in \ell^2(\Lambda)$. So we can use the eigenfunctions $\xi$ of $H_\Lambda$ to expand $\tilde{\phi}$ as follows
\begin{equation}
  \tilde{\phi}=\sum(\tilde{\phi},\xi)\xi.
\end{equation}
Next, we want to replace $\xi$ by $\xi'$(\ref{pri}). Let $Q=\{\xi'|\xi~ \text{satisfies}~ (\ref{xid2})\}$, we have
\begin{equation}
  \|\tilde{\phi}-\sum_{\xi'\in Q}(\tilde{\phi},\xi')\xi'\|_{\ell^2(\Lambda)}\leq O\left\{\mathrm{e}^{-(\log T)^{\frac{\sigma''-\delta}{\alpha}}}\sqrt{|
  \Lambda|}\|\hat{\phi}\|_{\ell^2}\right\}
\end{equation}
Since $0<s\leq \log T$,
\begin{equation}\label{LamNo.}
   |\Lambda|\leq T^{10s+2}\leq \mathrm{e}^{10(\log T)^2}\quad \text{and} \quad \frac{\sigma''-\delta}{\alpha}>2,
\end{equation}
 we have
\begin{equation}
  \|\tilde{\phi}-\sum_{\xi'\in Q}(\tilde{\phi},\xi')\xi'\|_{\ell^2(\Lambda)}\leq \mathrm{e}^{-\frac{2}{3}(\log T)^{\frac{\sigma''-\delta}{\alpha}}},
\end{equation}
By Plancherel's identity on $\mathbb{T}\times\mathbb{T}_T$, we have
\begin{equation}
  \begin{split}
    & \left\| \mathcal{F}^{-1}\left(\tilde{\phi}-\sum_{\xi'\in Q}(\tilde{\phi},\xi')\xi'\right)(x,\theta) \right\|_{L^2(\mathbb{T}\times\mathbb{T}_T)}\\
    =& \left\| \mathcal{F}^{-1}(\tilde{\phi})(x,\theta)-\sum_{\xi'\in Q}(\tilde{\phi},\xi')\sum_{(j,n)\in \Omega'}\xi'(j,n)\mathrm{e}^{\mathrm{i}(jx+\frac{n}{T}\theta)} \right\|_{L^2(\mathbb{T}\times\mathbb{T}_T)}
  \end{split}
\end{equation}
where $\mathbb{T}_T$ denotes $
\left[-\pi T,\pi T\right)$ with periodic boundary.

So $\phi(x):=\phi(x,0)$ as a function on $L^2(\mathbb{T})$ satisfies
\begin{equation}\label{md1}
  \begin{split}
     & \left\| \phi(x)-\sum_{\xi'\in Q}(\tilde{\phi},\xi')\sum_{(j,n)\in \Omega'}\xi'(j,n)\mathrm{e}^{\mathrm{i}jx} \right\|_{L^2(\mathbb{T})}\\
     & \leq T^{\frac{1}{2}}(\log T)^{\frac{\sigma}{2}}\mathrm{e}^{-\frac{2}{3}(\log T)^{\frac{\sigma''-\delta}{\alpha}}}\\
     & \leq \mathrm{e}^{-\frac{1}{2}(\log T)^{\frac{\sigma''-\delta}{\alpha}}}, \quad(\sigma''>2\alpha+\delta>2)
  \end{split}
\end{equation}
Therefore for $|t|\leq T$, the definition of $\check{\xi}$, \quad (See(\ref{appf}))
\begin{equation*}
  \check{\xi'}(x,t):=\sum_{(j,n)\in\Omega'}\xi'(j,n)\mathrm{e}^{\mathrm{i}(jx+\frac{n}{T}t)}
\end{equation*}
\textbf{ Lemma 3.2.}, (\ref{LamNo.}) and (\ref{md1}) give
\begin{equation}\label{error}
  \begin{split}
     & \left\| S(t)\phi-\sum_{\xi'\in Q}(\tilde{\phi},\xi')\mathrm{e}^{\mathrm{i}Et}\check{\xi'}(t)\right\|_{L^2(\mathbb{T})} \\
      \leq & \left\| S(t)\phi-S(t)\sum_{\xi'\in Q}(\tilde{\phi},\xi')\sum_{(j,n)\in \Omega'}\xi'(j,n)\mathrm{e}^{\mathrm{i}jx}  \right\|_{L^2(\mathbb{T})}\\
     + &\left\|S(t)\sum_{\xi'\in Q}(\tilde{\phi},\xi')\sum_{(j,n)\in \Omega'}\xi'(j,n)\mathrm{e}^{\mathrm{i}jx}-\sum_{\xi'\in Q}(\tilde{\phi},\xi')\mathrm{e}^{\mathrm{i}Et}\check{\xi'}(t) \right\|_{L^2(\mathbb{T})}\\
     \leq & \left\| \phi(x)-\sum_{\xi'\in Q}(\tilde{\phi},\xi')\sum_{(j,n)\in \Omega'}\xi'(j,n)\mathrm{e}^{\mathrm{i}jx} \right\|_{L^2(\mathbb{T})}\\
     + & \sum_{\xi'\in Q}|(\tilde{\phi},\xi')|\left\|\mathrm{e}^{\mathrm{i}Et}\check{\xi'}(t)-S(t)\check{\xi'}(0)\right\|_{L^2(\mathbb{T})}\\
     \leq & \mathrm{e}^{-\frac{1}{2}(\log T)^{\frac{\sigma''-\delta}{\alpha}}}+|\Lambda|T\mathrm{e}^{-\frac{7}{12}(\log T)^{\frac{\sigma''-\delta}{\alpha}}}\\
     \leq & 2\mathrm{e}^{-\frac{1}{2}(\log T)^{\frac{\sigma''-\delta}{\alpha}}}.\quad(\sigma''>2\alpha+\delta)
  \end{split}
\end{equation}
Then we only need to estimate $\|\sum_{\xi'\in Q}(\tilde{\phi},\xi')\mathrm{e}^{\mathrm{i}Et}\check{\xi'}\|_{H^s}\quad(s>0)$, since intermediate frequencies satisfies
\begin{equation}
  |j|\leq J=T^{10s}=\mathrm{e}^{10s\log T}\leq \mathrm{e}^{10(\log T)^2}.
\end{equation}

Let $\widehat{\widehat{\mathrm{e}^{\mathrm{i}Et}\check{\xi'}}}$ be the Fourier transform of $\mathrm{e}^{\mathrm{i}Et}\check{\xi'}$ with respect to $x$.
\begin{equation}\label{md2}
  \begin{split}
     & \left\|\sum_{\xi'\in Q}(\tilde{\phi},\xi')\mathrm{e}^{\mathrm{i}Et}\check{\xi'}\right\|_{H^s}=\left[ \sum_j|j|^{2s}\left|\sum_{\xi'\in Q}(\tilde{\phi},\xi')\widehat{\widehat{\mathrm{e}^{\mathrm{i}Et}\check{\xi'}}}(j)\right|^2 \right]^{\frac{1}{2}}\\
     = & \left[ \sum_j|j|^{2s}\left|\sum_{k}\sum_{\xi'\in Q}\hat{\phi}(k)\xi'(k,0)\widehat{\widehat{\mathrm{e}^{\mathrm{i}Et}\check{\xi'}}}(j)\right|^2 \right]^{\frac{1}{2}}.
  \end{split}
\end{equation}
From the support of $\xi'$ (\ref{xid2}),
\begin{equation}
  ||j|-|k||\leq 2(\log T)^{\sigma}\quad(\sigma>2).
\end{equation}
Since
\begin{equation}\label{jg}
  |j|>2J_0=8D(\log T)^\sigma \quad (D>2\pi)
\end{equation}
from (\ref{dj0}) (\ref{xid2}) (\ref{jg}) imply
\begin{equation}
  |j|/2<|k|<2|j|.
\end{equation}

We now make a dyadic decomposition of $\phi$. Let $R$ be dyadic and
\begin{equation}
  R/2<|j|<2R.
\end{equation}
So
\begin{equation}
  R/4<|k|<4R.
\end{equation}
Let
\begin{equation}
  \phi_R=\sum_{R/4<|k|<4R}\hat{\phi}(k)\mathrm{e}^{\mathrm{i}kx}.
\end{equation}
We then have
\begin{equation}\label{hs}
  \begin{split}
      (\ref{md2})&\leq \left[ \sum_{R~\text{dyadic}}4^s R^{2s}\sum_{R/2<|j|<2R}\left|\sum_{k}\sum_{\xi'\in Q}\hat{\phi}_R(k)\xi'(k,0)\widehat{\widehat{\mathrm{e}^{\mathrm{i}Et}\check{\xi'}}}(j)\right|^2 \right]^{\frac{1}{2}} \\
     & \leq  \left[ \sum_{R~\text{dyadic}}4^s R^{2s}\left\|\sum_{\xi'\in Q}(\tilde{\phi}_R,\xi')\mathrm{e}^{\mathrm{i}Et}\check{\xi'}\right\|_{L^2(\mathbb{T})}^2 \right]^{\frac{1}{2}}
  \end{split}
\end{equation}
Using (\ref{error}) and since $ \mathrm{supp}\phi_R\subset \mathrm{supp} \phi \subseteq[-J/2,-2J_0]\cup[2J_0,J/2]$,
\begin{equation}\label{l2}
  \begin{split}
     \left\| \sum_{\xi'\in Q}(\tilde{\phi}_R,\xi')\mathrm{e}^{\mathrm{i}Et}\check{\xi'} \right\|_{L^2(\mathbb{T})}&\leq \|S(t)\phi_R\|_{L^2(\mathbb{T})}+2\mathrm{e}^{-\frac{1}{2}(\log T)^{\frac{\sigma''-\delta}{\alpha}}}\|\phi_R\|_{L^2(\mathbb{T})}\\
    & \leq 2\|\phi_R\|_{L^2(\mathbb{T})}\quad(\sigma''>2\alpha+\delta>2).
  \end{split}
\end{equation}
Using (\ref{l2}) in (\ref{hs}), we have
\begin{equation}\label{md3}
  \left\| \sum_{\xi'\in Q}(\tilde{\phi}_R,\xi')\mathrm{e}^{\mathrm{i}Et}\check{\xi'} \right\|_{H^s(\mathbb{T})}\leq \left[ \sum_{R~\text{dyadic}}4^s R^{2s}\cdot4\|\phi_R\|^2_{L^2(\mathbb{T})} \right]^{\frac{1}{2}}\leq C^s\|\phi\|_{H^s}.
\end{equation}
Combining (\ref{md3}) with (\ref{error}), (\ref{TD}), we obtain (\ref{middle}) with a slightly large $C$.
\end{proof}

\section{Estimate of Sobolev Norms on High Frequency}

In order to estimate Sobolev norm on high frequency
\begin{equation}\label{highf}
  \|\Pi_{J/4}S(t)(I-\Pi_{J/2})u_0\|_{H^s}+\|(1-\Pi_{J/4})S(t)u_0\|_{H^s}
\end{equation}
in (\ref{cutoff}), we will use the following lemma
\begin{lem}
For the initial datum $u_0\in H^{s+1}(\mathbb{T})\cap H^{s}(\mathbb{T})$ satisfying the condition (\ref{scon}), and $V$ satisfying  (\ref{grcon2}), we have
\begin{equation}\label{h1}
  \left\|\left[\frac{\partial^\gamma V}{\partial x^\gamma},\Pi_J\right]\right\|_{H^{s-\gamma}\rightarrow H^{s-\gamma}}\leq \frac{Cs!}{J}\quad(J\gg1).
\end{equation}
\begin{equation}\label{h2}
  \|(I-\Pi_J)S(t)\|_{H^s\rightarrow H^s}\leq C|t|+\frac{(C^s(s+1)!)^2}{J}|t|^{s+2}\quad(J>|t|^s),
\end{equation}
\begin{equation}\label{h3}
  \|[S(t),\Pi_J]\|_{H^s\rightarrow H^s}\leq \frac{(C^s(s+1)!)^4}{J}(|t|^{3s+2}+1)\quad(J>|t|^s).
\end{equation}
\end{lem}
\begin{proof}
Let $\widehat{\widehat{V}}$ and $\widehat{\widehat{u}}$ be the partial Fourier transform with respect to $x$. Since
\begin{equation}
  \widehat{\widehat{[V,\Pi_J]u}}=\widehat{\widehat{V}}*\widehat{\Pi}_J\widehat{\widehat{u}}-\widehat{\Pi}_J\widehat{\widehat{V}}*\widehat{\widehat{u}},
\end{equation}
we have
\begin{equation}
  [V,\Pi_J]^{\widehat{}}(j,j')=\widehat{\widehat{V}}(j-j')(\widehat{\Pi}_J(j')-\widehat{\Pi}_J(j)),
\end{equation}
where $\widehat{\Pi}_J$ is defined in (\ref{Pi}). Since $V$ is analytic, periodic in $x$ and $|V(x,t)|<C$ for all $t$,
\begin{equation*}
  |\widehat{\widehat{V}}(j-j')|\leq C\mathrm{e}^{-c|j-j'|},
\end{equation*}
and from (\ref{Pi})
\begin{equation}\label{pipro}
  \begin{split}
     |\widehat{\Pi}_J(j')-\widehat{\Pi}_J(j)|&\leq 1,\quad |j-j'|\geq J/2, \\
     &\leq\frac{2}{J}|j-j'|,\quad|j-j'|<J/2.
  \end{split}
\end{equation}
Using (\ref{pipro}), we have
\begin{equation}
  \begin{split}
     |[V,\Pi_J]^{\widehat{}}(j,j')|&\leq C\mathrm{e}^{-c|j-j'|},\quad|j-j'|\geq J/2\\
     & \leq \frac{2C}{J}|j-j'|\mathrm{e}^{-c|j-j'|},\quad |j-j'|<J/2.
  \end{split}
\end{equation}
From Schur's lemma, we then obtain
\begin{equation}
  \|[V,\Pi_J]\|_{H^s\rightarrow H^s}\leq \frac{Cs!}{J}\quad (J\gg 1).
\end{equation}
It follows that
\begin{equation}\label{vcon}
  \left\|\left[\frac{\partial^\gamma V}{\partial x^\gamma},\Pi_J\right]\right\|_{H^{s-\gamma}\rightarrow H^{s-\gamma}}\leq \frac{Cs!}{J}\quad (J\gg 1).
\end{equation}
Then
\begin{equation}
  \begin{split}
     \frac{\partial}{\partial t}\|(I-\Pi_J)u(t)\|^2_{H^s}&=2\mathrm{Re}\left((I-\Pi_J)\frac{\partial^s}{\partial x^s}u(t),-(I-\Pi_J)\frac{\partial^s}{\partial x^s}\left(\frac{1}{2}V_x(x,t)u+V(x,t)u_x\right)\right)  \\
     &=2\mathrm{Re}\left((I-\Pi_J)\frac{\partial^s}{\partial x^s}u(t),-\frac{3}{2}(I-\Pi_J)V_x\frac{\partial^s u}{\partial x^s}-(I-\Pi_J)V\frac{\partial^{1+s}u}{\partial x^{1+s}}\right)  \\
     &\quad+2\mathrm{Re}\left((I-\Pi_J)\frac{\partial^s}{\partial x^s}u(t),-\frac{1}{2}(I-\Pi_J)\sum_{\substack{\gamma+\beta=s\\\gamma\geq 1}}\frac{\partial^{1+\gamma}V}{\partial x^{1+\gamma}}\frac{\partial^\beta u}{\partial x^\beta}\right)\\
     &\quad+2\mathrm{Re}\left((I-\Pi_J)\frac{\partial^s}{\partial x^s}u(t),-(I-\Pi_J)\sum_{\substack{\gamma+\beta=s\\\gamma\geq 2}}\frac{\partial^\gamma V}{\partial x^\gamma}\frac{\partial^{1+\beta}u}{\partial x^{1+\beta}}\right).\\
  \end{split}
\end{equation}
It follows that
\begin{equation}
  \begin{split}
     &\mathrm{Re}\left((I-\Pi_J)\frac{\partial^s}{\partial x^s}u(t),-\frac{3}{2}(I-\Pi_J)V_x\frac{\partial^s u}{\partial x^s}-(I-\Pi_J)V\frac{\partial^{1+s}u}{\partial x^{1+s}}\right)  \\
     \leq & \left|\mathrm{Re} \left((I-\Pi_J)\frac{\partial^s u}{\partial x^s},-\frac{3}{2}V_x(I-\Pi_J)\frac{\partial^s u}{\partial x^s}-V(I-\Pi_J)\frac{\partial^{1+s}u}{\partial x^{1+s}}\right) \right|\\
     +&\|(1-\Pi_J)u(t)\|_{H^s}\left(\frac{3}{2}[V_x,\Pi_J]\|u(t)\|_{H^s}+[V,\Pi_J]\|u\|_{H^{s+1}}\right)\\
     \leq &\|V_x\|_{L^\infty}\|(I-\Pi_J)u\|_{H^s}^2+\|(I-\Pi_J)u\|_{H^s}\frac{(C^s(s+1)!)^2(|t|^{s+1}+1)}{J},
  \end{split}
\end{equation}
where we use integration by part, (\ref{vcon}) and (\ref{scon}).

Since $V$ is real analytic in $(x,t)$ and bounded in $D$, we have
\begin{equation}
  \left\|\frac{\partial^m{V}}{\partial x^m}\right\|_{L^\infty(\mathbb{T})}\leq C^{m+1}m!,\quad m=0,1,\ldots
\end{equation}
Using this property, (\ref{vcon}) and
\begin{equation}
  \|(I-\Pi_J)u(t)\|_{H^{s-\gamma}}\leq \frac{1}{J^\gamma}\|u(t)\|_{H^s},
\end{equation}
we obtain
\begin{equation}
  \begin{split}
     &\mathrm{Re}\left((I-\Pi_J)\frac{\partial^s}{\partial x^s}u(t),-\frac{1}{2}(I-\Pi_J)\sum_{\substack{\gamma+\beta=s\\\gamma\geq 1}}\frac{\partial^{1+\gamma}V}{\partial x^{1+\gamma}}\frac{\partial^\beta u}{\partial x^\beta}\right)\\
     +&\mathrm{Re}\left((I-\Pi_J)\frac{\partial^s}{\partial x^s}u(t),-(I-\Pi_J)\sum_{\substack{\gamma+\beta=s\\\gamma\geq 2}}\frac{\partial^\gamma V}{\partial x^\gamma}\frac{\partial^{1+\beta}u}{\partial x^{1+\beta}}\right)\\
     \leq & \|(I-\Pi_J)u(t)\|_{H^{s}}\left(\left[\frac{\partial^2V}{\partial x^2},\Pi_J\right]\|u(t)\|_{H^{s-1}}+C^32!\|(I-\Pi_J)u(t)\|_{H^{s-1}}\right)\\
     + &\ldots + \|(I-\Pi_J)u(t)\|_{H^{s}}\left(\left[\frac{\partial^{1+\gamma}V}{\partial x^{1+\gamma}},\Pi_J\right]\|u(t)\|_{H^{s-\gamma}}+C^{\gamma+2}(\gamma+1)!\|(I-\Pi_J)u(t)\|_{H^{s-\gamma}}\right)\\
     + &\ldots + \|(I-\Pi_J)u(t)\|_{H^{s}}\left(\left[\frac{\partial^{1+s}V}{\partial x^{1+s}},\Pi_J\right]\|u(t)\|_{L^2}+C^{s+2}(s+1)!\|(I-\Pi_J)u(t)\|_{L^2}\right)\\
     \leq & \|(I-\Pi_J)u(t)\|_{H^{s}}\frac{(C^{s}(s+1)!)^2(|t|^s+1)}{J}.
  \end{split}
\end{equation}
With the help of (\ref{grcon2}), we have the estimate (\ref{h2}).

To prove (\ref{h3}), assume $u$ is a solution to (\ref{equ})
\begin{equation*}
  u_t+u_{xxx}+\frac{1}{2}V_x(x,t)u+V(x,t)u_x=0,
\end{equation*}
then
\begin{equation}
  \left(\partial_t+\partial_{xxx}\right)\Pi_Ju+\frac{1}{2}\partial_xV\Pi_Ju+V\partial_x\Pi_Ju=\frac{1}{2}[\partial_xV,\Pi_J]u+[V,\Pi_J]u_x
\end{equation}
From \textbf{Lemma 2.3.}
\begin{equation}
  [S(t),\Pi_J]u_0=S(t)\Pi_Ju_0-\Pi_JS(t)u_0=\int_0^tS(t)S(\tau)^{-1}\left(\frac{1}{2}[\partial_xV,\Pi_J]u+[V,\Pi_J]u_x\right)\mathrm{d}\tau,
\end{equation}
Using \textbf{Lemma 2.2.}, (\ref{h1}) and (\ref{scon}), we obtain (\ref{h3}).
\end{proof}

\section{Iteration and Interpolation}
We use the following two decompositions. The first one decomposes into low (\ref{low}) and high frequencies (\ref{high}).
\begin{eqnarray}
  \|S(t)u_0\|_{H^s}\leq&&\|\Pi_{J/4}S(t)u_0\|_{H^s}\label{low}\\
  +&&\|(1-\Pi_{J/4})S(t)u_0\|_{H^s}\label{high}
\end{eqnarray}
The second one decomposes into sub-low (\ref{sub-l}), sub-intermediate (\ref{sub-m}) and sub-high (\ref{sub-h}) frequencies.
\begin{eqnarray}
  \|\Pi_{J/4}S(t)u_0\|_{H^s}\leq&&\|\Pi_{J/4}S(t)\Pi_{2J_0}u_0\|_{H^s}\label{sub-l}\\
 +&&\|\Pi_{J/4}S(t)(\Pi_{J/2}-\Pi_{2J_0})u_0\|_{H^s}\label{sub-m}\\
 +&&\|\Pi_{J/4}S(t)(I-\Pi_{J/2})u_0\|_{H^s}\label{sub-h}.
\end{eqnarray}

Since the choice of $J$
\begin{equation*}
  J=T^{10s},
\end{equation*}
we can use estimate (\ref{h2}) to control high frequencies (\ref{high})
\begin{equation}\label{hest}
  \|(1-\Pi_{J/4})S(T)u_0\|_{H^s}\leq C|T|
\end{equation}
for a large $T$ which we determined later.

Then (\ref{middle}) controls sub-intermediate(\ref{sub-m}) and (\ref{h3}) controls sub-high (\ref{sub-h}) frequencies
\begin{equation}\label{siest}
  \begin{split}
      \|\Pi_{J/4}S(T)(\Pi_{J/2}-\Pi_{2J_0})u_0\|_{H^s}& \leq C^s\|(\Pi_{J/2}-\Pi_{2J_0})u_0\|_{H^s}\\
     & \leq C^s\|u_0\|_{H^s}\leq C^s,
  \end{split}
\end{equation}
\begin{equation}\label{shest}
  \begin{split}
      \|\Pi_{J/4}S(T)(I-\Pi_{J/2})u_0\|_{H^s}&\leq \|\Pi_{J/4}(I-\Pi_{J/2})S(t)u_0\|_{H^s}+\|\Pi_{J/4}[S(T),\Pi_{J/2}]u_0\|_{H^s}\\
     &\leq \|[S(T),\Pi_{J/2}]u_0\|_{H^s}\\
     & \leq \frac{(C^s(s+1)!)^4}{J}(|T|^{3s+2}+1)\|u_0\|_{H^s} \leq 1,
  \end{split}
\end{equation}
where $ \|\Pi_{J/4}(I-\Pi_{J/2})S(t)u_0\|_{H^s}=0.$

So the only work left is to control (\ref{sub-l}), the sub-low frequencies: $|j|\leq 2J_0$, which we do by iterating $S(0,T):=S(T)$, $|T|$ times and each time making again the decomposition as in (\ref{sub-l})-(\ref{sub-h}), using the same estimate as above.

We have
\begin{eqnarray}
  \|\Pi_{J/4}S(0,t)\Pi_{2J_0}u_0\|_{H^s}\leq&&\|\Pi_{J/4}S(1,t)\Pi_{2J_0}S(0,1)\Pi_{2J_0}u_0\|_{H^s}\label{subl-l}\\
 +&&\|\Pi_{J/4}S(1,t)(\Pi_{J/2}-\Pi_{2J_0})S(0,1)\Pi_{2J_0}u_0\|_{H^s}\label{subl-m}\\
 +&&\|\Pi_{J/4}S(1,t)(I-\Pi_{J/2})S(0,1)\Pi_{2J_0}u_0\|_{H^s}\label{subl-h}.
\end{eqnarray}
which is the analogue at $t=1$ of the decomposition in (\ref{sub-l})-(\ref{sub-h}), with $S(0,1)\Pi_{2J_0}u_0$ replacing $u_0$. So we have
\begin{equation}\label{itm}
  (\ref{subl-m})\leq C^s\|S(0,1)\Pi_{2J_0}u_0\|_{H^s}\leq 2C^{2s}(s+1)!,
\end{equation}
where we used
\begin{equation}
  \|S(0,1)\|_{H^s\rightarrow H^s}\leq 2C^s(s+1)!
\end{equation}
from \textbf{Lemma 2.2.}. With the same estimate above, we have
\begin{equation}\label{ith}
  (\ref{subl-h})\leq \frac{(C^s(s+1)!)^4}{J}(|T|^{3s+2}+1)\|S(0,1)\Pi_{2J_0}u_0\|_{H^s}\leq 2C^{2s}(s+1)!.
\end{equation}
(\ref{itm}),(\ref{ith}) are the analogues of (\ref{siest}) and (\ref{shest}), which control sub-intermediate and sub-high frequencies.

If we call (\ref{subl-l})£¬(\ref{subl-m}) and (\ref{subl-h}) as \textbf{1-sub-low}, \textbf{1-sub-intermediate} and \textbf{1-sub-high} frequencies, then using (\ref{itm}),(\ref{ith}), we have after one iteration:
\begin{equation}
\begin{split}
   \|\Pi_{J/4}S(0,t)\Pi_{2J_0}u_0\|_{H^s}& \leq \|\Pi_{J/4}S(1,t)\Pi_{2J_0}S(0,1)\Pi_{2J_0}u_0\|_{H^s} + 4C^{2s}(s+1)!\\
   &= \text{\textbf{1-sub-low}}+4C^{2s}(s+1)!.
\end{split}
\end{equation}
After $r$ iterations, \textbf{$r$-sub-intermediate} satisfies
\begin{equation}
  \begin{split}
     &\textbf{$r$-sub-intermediate}\\
     &=\|\Pi_{J/4}S(r,t)(\Pi_{J/2}-\Pi_{2J_0})S(r-1,r)\Pi_{2J_0}S(r-2,r-1)\Pi_{2J_0}\ldots\Pi_{2J_0}u_0\|_{H^s}\\
     &\leq C^s\|S(r-1,r)\Pi_{2J_0}S(r-2,r-1)\Pi_{2J_0}\ldots\Pi_{2J_0}u_0\|_{H^s}\\
     &\leq C^s\|S(r-1,r)\|_{H^s\rightarrow H^s}\cdot(2J_0)^s\\
     &\leq 2C^{2s}(s+1)!(2J_0)^s;
  \end{split}
\end{equation}
while \textbf{$r$-sub-high} satisfies
\begin{equation}
  \begin{split}
     \textbf{$r$-sub-high}&=\|\Pi_{J/4}S(r,t)(I-\Pi_{J/2})S(r-1,r)\Pi_{2J_0}S(r-2,r-1)\Pi_{2J_0}\ldots\Pi_{2J_0}u_0\|_{H^s}\\
     &\leq \frac{(C^s(s+1)!)^4}{J}(|T|^{3s+2}+1)\|S(r-1,r)\Pi_{2J_0}S(r-2,r-1)\Pi_{2J_0}\ldots\Pi_{2J_0}u_0\|_{H^s}\\
     & \leq\|S(r-1,r)\|_{H^s\rightarrow H^s}\cdot(2J_0)^s\\
     & \leq 2C^{s}(s+1)!(2J_0)^s.
  \end{split}
\end{equation}
After $|T|$ iterations, we then have
\begin{equation}\label{Test}
  \begin{split}
     \|\Pi_{J/4}S(0,t)\Pi_{2J_0}u_0\|_{H^s}&\leq \textbf{$|T|$-sub-low}+4C^{2s}(s+1)!(2J_0)^s|T|\\
     & \leq \|\Pi_{2J_0}S(T-1,T)\Pi_{2J_0}\ldots \Pi_{2J_0}S(r-1,r)\Pi_{2J_0}\ldots\Pi_{2J_0}u_0\|_{H^s}\\
     & \quad +4C^{2s}(s+1)!(2J_0)^s|T|\\
     & \leq |T|(sJ_0)^{s}\cdot C^s
  \end{split}
\end{equation}
with a larger $C$.

Using (\ref{Test}) in (\ref{sub-l}) and combining with (\ref{siest}),(\ref{shest}), we obtain
\begin{equation}\label{Tlest}
  \|\Pi_{J/4}S(0,T)u_0\|_{H^s}\leq |T|(sJ_0)^sC^s.
\end{equation}
Using (\ref{Tlest}) in (\ref{low}), we have
\begin{equation}
     \|S(0,T)u_0\|_{H^s}\leq C^s|T|(sJ_0)^{s}
\end{equation}
for all $0<s\leq \log T$. Interpolating with the $L^2$ bound $\|S(t)\|_{L^2\rightarrow L^2}= 1$(\ref{con}) yields
\begin{equation}
  \begin{split}
    \|S(0,T)\|_{H^{s'}\rightarrow H^{s'}} &  \leq |T|^{\frac{s'}{s}}(CsJ_0)^{s'} \\
     & \leq C^{s'}(\log T)^{(\sigma+1)s'}\quad(\sigma>2)
  \end{split}
\end{equation}
with a larger $C$, for all $0<s'<s$, where we fixed $s=\log|T|$ and used (\ref{dj0}).

For a fixed $s>0$, for $|t|<\mathrm{e}^s$, \textbf{Lemma 2.2} gives
\begin{equation*}
  \|S(0,t)\|_{H^{s'}\rightarrow H^{s'}}\leq C^s(s+1)!e^{s^2},
\end{equation*}
for $|t|>\mathrm{e}^s$, we use (\ref{Tlest}). This gives immediately
\begin{equation}
  \|S(0,t)\|_{H^s\rightarrow H^s}\leq C_s(\log(|t|+2))^{(\sigma+1)s}
\end{equation}
for all $s>0$. Let $\varsigma=\sigma+1$, we obtain\textbf{ Theorem 1.1.}

\end{document}